\documentclass[a4paper,12pt]{article}

\usepackage{amsmath}
\usepackage{amsthm}
\usepackage{graphicx}
\usepackage{verbatim}
\usepackage{epsfig}
\usepackage{hyperref}
\usepackage{float}
\usepackage{color}
\usepackage{fullpage}
\usepackage{amsfonts}
\usepackage{amscd}
\usepackage{amssymb}
\usepackage{graphics}
\usepackage{amsmath}
\usepackage[english]{babel}
\usepackage{bbm}
\usepackage{yfonts}
\usepackage{mathrsfs}
\usepackage{color}

\DeclareFontFamily{OML}{rsfs}{\skewchar\font'177}
\DeclareFontShape{OML}{rsfs}{m}{n}{ <5> <6> rsfs5 <7> <8> <9> rsfs7
  <10> <10.95> <12> <14.4> <17.28> <20.74> <24.88> rsfs10 }{}
\DeclareMathAlphabet{\mathfs}{OML}{rsfs}{m}{n}

\newcommand{\BN}{{\mathbb{N}}}

\newcommand{\BZ}{{\mathbb{Z}}}

\newcommand{\CN}{{\mathcal{N}}}

\newcommand{\ind}{{\mathbbm{1}}}

\newcommand{\prob}{{\bf P}}
\newcommand{\bae}{\begin{equation}\begin{aligned}}
\newcommand{\eae}{\end{aligned}\end{equation}}
\newcommand{\ev}{\mathbf{E}}
\newcommand{\pr}{\mathbf{P}}
\newcommand{\Z}{\mathbb{Z}}
\newcommand{\om}{{\omega}}

\newtheorem{thm}{Theorem}[section]
\newtheorem{prop}[thm]{Proposition}
\newtheorem{lem}[thm]{Lemma}
\newtheorem{claim}[thm]{Claim}
\newtheorem{cor}[thm]{Corollary}
\newtheorem{definition}{Definition}[section]

\newtheorem{rmk}{Remark}

\begin{document}

\numberwithin{equation}{section}
\numberwithin{figure}{section}
\title{Mutually excited random walks}
\author{Noam Berger\footnote{Hebrew University of Jerusalem and Technische Universit\"at M\"unchen}, Eviatar B. Procaccia\footnote{Weizmann Institute of Science}\footnote{Research supported by ISF grant 1300/08 and EU grant PIRG04-GA-2008-239317 }}
\maketitle
\abstract{Consider two random walks on $\mathbb{Z}$. The transition
probabilities of each walk is dependent on trajectory of
the other walker i.e. a drift $p>1/2$ is obtained in a position the other walker visited twice or more. This simple model has a speed which is, according to simulations, not monotone in $p$, without apparent ``trap" behaviour. In this paper we prove the process has positive speed for $1/2<p<1$, and present a deterministic algorithm to approximate the speed and show the non-monotonicity. }

\section{Introduction}
Excited random walk or "cookie motion" is a well known model in probability theory introduced by Benjamini and Wilson in \cite{benjamini2003excited}. A random walk on $\Z^d$ is excited if the first time it visits a vertex there is a bias in one direction,
but on subsequent visits to that vertex the walker picks a neighbor uniformly at random. Benjamini and Wilson proved that excited random walk on $\Z^d$ is transient iff $d > 1$. Benjamini and wilson also proved that for $d\geq4$ (later proved by Gadi Kozma for $d\le2$ \cite{kozma2005excited}) that excited random walk has linear speed. That is
\[
\liminf_{n\rightarrow\infty}\frac{X(n)_1}{n}>0,
\]
where $X(n)$ is the position of the walk at time $n$ and $X(n)_1$ is the projection of $X(n)$ to the first coordinate (the direction of the bias). Kozma proved in \cite{kozma2003excited} that in three dimensions excited random walk has positive speed. That is
\[
\lim_{n\rightarrow\infty}\frac{X(n)}{n}>0
.\]

Martin P. W. Zerner defined in \cite{ZERMERW} a generalization of the Benjamini-Wilson model. Zerner defines a cookie environment as an element \[\omega=((\omega(z,i))_{i>0})_{z\in\Z}\in\Omega_+=\left([1/2,1]^\mathbb{N}\right)^\Z,\]where $\omega(z,i)$ is the probability for the random walk to jump from $z$ to $z+1$ if it is currently
visiting $z$ for the $i$-th time. In that paper Zerner proved monotonicity of hitting times with respect to the environment (Lemma 15 of \cite{ZERMERW}), i.e let $\omega_1,\omega_2\in\Omega_+$ with $\omega_1\leq\omega_2$ and $-\infty\leq x\leq y\leq z\leq \infty$, $y\in\Z$, $x,z\in\Z\cup\{\infty\}$ and $t\in\mathbb{N}\cup\{\infty\}$. Then
\[
\pr_{y,\omega_1}[T_z\leq T_x\wedge t]\leq\pr_{y,\omega_2}[T_z\leq T_x\wedge t]
.\]
Using the above lemma, Zerner proved monotonicity of speed with respect to the environment (Theorem 17 of \cite{ZERMERW}), i.e let $\overline{\pr}$ be a stationary and ergodic probability measure on $\Omega_+^2$ such that $\overline{\pr}(\omega_1\leq\omega_2)=1$ then $v_1\leq v_2$, where $v_i$ is the $\pr_{0,\omega_i}$ a.s limit of $\frac{X_n}{n}$, $i\in\{1,2\}$, under the assumption that both limits exist and are a.s constants.

Recently Itai Benjamini proposed a new model. Consider two random walks $(X_n,Y_n)$ on $\mathbb{Z}$ each
eating its own favorite kind of cookie. The transition
probabilities of each walk is dependent on the number of cookies
the other walker has eaten. The cookie environment is an element
\bae{\omega_1\choose\omega_2}\in\left[\left(\left[\frac{1}{2},1\right]^{\mathbb{N}\cup\{0\}}\right)^\mathbb{Z}\right]^2\eae
satisfying
\bae{\omega_1(z)\choose\omega_2(z)}={(\frac{1}{2},\frac{1}{2},p,p,p,...)\choose
(\frac{1}{2},\frac{1}{2},p,p,p,...)}\eae

Define $\mathfs{F}_n=\sigma\left(X_{0\leq m\leq n},Y_{0\leq m\leq
n}\right)$, the mutually excited random walk (MERW) is defined by
    \bae\label{eq:modelrule}
    \pr[X_{n+1}=X_n+1|\mathfs{F}_n]&=\omega_1(X_n,\#\{m\leq n:Y_m=X_n\})\\
    \pr[X_{n+1}=X_n-1|\mathfs{F}_n]&=1-\omega_1(X_n,\#\{m\leq
    n:Y_m=X_n\})\\
    \pr[Y_{n+1}=Y_n+1|\mathfs{F}_n]&=\omega_2(Y_n,\#\{m\leq n:X_m=Y_n\})\\
    \pr[Y_{n+1}=Y_n-1|\mathfs{F}_n]&=1-\omega_2(Y_n,\#\{m\leq
    n:X_m=Y_n\})
    \eae

The main interest in Benjamini's model arises from the fact that simulation results indicate that the limiting speed is a non monotone function of the drift parameter $p$ (see Figure \ref{fig:merwspeed}, created by computer simulation). This stands in contrast to Zerner's result that for self excited random walk the limiting speed is monotone.

In this paper, we prove that a MERW is transient and has positive speed for all $p\in(1/2,1]$ (Theorem \ref{thm:converspeed}). Using approximation by a family of Markov processes we prove that $v(p)$ is continuous in $[1/2,1)$ (Lemma \ref{lem:vkcon}),
where $v(p)$ is the speed of a MERW with drift $p$. We also present a scheme for a computer assisted proof of the non monotonicity.
At present time we do not have a conceptual proof for the non monotonicity in this model.
\begin{figure}
\begin{center}
\includegraphics[width=0.5\textwidth]{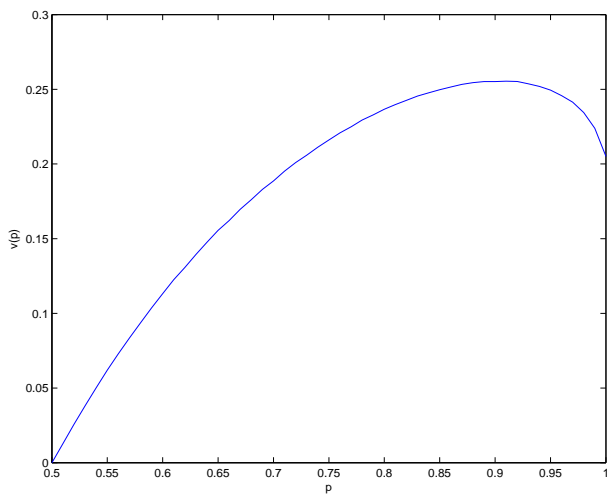}
\caption{Simulation results for the speed of MERW.\label{fig:merwspeed}}
\end{center}
\end{figure}

\section{RWRE preliminaries}
\subsection{RWRE facts}
In this section we give a short introduction, based on \cite{ZEI}, to one dimensional random walk in random environment (RWRE) and state its use in this paper. Let $\mathfs{P}$ be the interval $[0,1]$ and set $\Omega=\mathfs{P}^\Z$. We endow $\Omega$ with the product $\sigma$-algebra and a product measure $\mathbb{P}$ (i.i.d RWRE). For any environment $\omega\in\Omega$ we define a $\Z$ valued Markov process $\{\xi_n\}_{n\geq0}$ endowed with the quenched measure $\pr_{z,\omega}$ satisfying,
\bae
\pr_{z,\omega}(\xi_0=z)&=1\\
\pr_{z,\omega}(\xi_{n+1}=\xi_n+1|\xi_0,\xi_1,\ldots,\xi_n)&=\omega(\xi_n)\quad \pr_{z,\omega}{\rm -a.s}\\
\pr_{z,\omega}(\xi_{n+1}=\xi_n-1|\xi_0,\xi_1,\ldots,\xi_n)&=1-\omega(\xi_n)\quad \pr_{z,\omega}{\rm -a.s}.
\eae
The annealed measure is defined on $\Omega\times \Z^\mathbb{N}$ by $\pr_z(\cdot)=\mathbb{E}[\pr_{z,\omega}(\cdot)]$. Define $\rho_x=\frac{1-\omega(x)}{\omega(x)}$, assume $\ev(\log\rho_0)$ is well defined and $1>\omega(x)>0$.

The following theorem is a special case of Theorem 2.1.9 of \cite{ZEI}.
\begin{thm}\label{thm:rwrespeed} If $\ev[\rho_0]<1$ then
\[
\lim_{n\rightarrow\infty}\frac{\xi_n}{n}=\frac{1-\ev[\rho_0]}{\ev\left[\frac{1}{\omega_0}\right]}
.\]
\end{thm}

\begin{lem}\label{lem:rwrebacktrack}
Let $\xi_n$ be an i.i.d RWRE starting at $0$, then the annealed probability of backtracking $k$ steps before reaching $+1$ is smaller than $2(\ev\rho_0)^k$.
\end{lem}
\begin{proof}
Let $\tau_i=\inf\{n|\xi_n=i\}$ and $v_\omega(z)=\pr_{z,\omega}(\tau_{-k}<\tau_1)$. We can obtain a recursion relation \cite{ZEI}
\bae\label{eq:reqrelatonrwre}
v_\omega(z)&=\omega_zv_\omega(z+1)+(1-\omega_z)v_\omega(z-1)\\
v_\omega(1)&=0,~v_\omega(-k)=1.
\eae
Iterating equation \eqref{eq:reqrelatonrwre}, yields
\bae
v_\omega(1)-v_\omega(0)&=\rho_0(v_\omega(0)-v_\omega(-1))=\rho_0\rho_{-1}[v_\omega(-1)-v_\omega(-2)]\\
&=\rho_0\rho_{-1}\cdots\rho_{-(k-1)}[v_\omega(-(k-1))-v_\omega(-k)]
\eae
Taking expectation over the environments and using independence we obtain
\[
v(0)=\ev v_\omega(0)=\pr_0(\tau_{-k}<\tau_1)\leq2(\ev\rho_0)^k
 \]
\end{proof}
\begin{cor}\label{cor:backtracking}
Let $\xi_n$ be an i.i.d RWRE starting at $0$. Then
\bae
\ev\left[\left(1-\pr_{z,\omega}(\tau_{-k}<\tau_1)\right)^n\right]\ge\left(1-2(\ev\rho_0)^k\right)^n.
\eae
\end{cor}
\begin{proof}
The inequality is Jensen's inequality.
\end{proof}
Another useful result in RWRE we will use in this paper is that large deviation probabilities for environments with only positive and zero drifts decay like $e^{-Cn^{1/3}}$. The following theorem, which is Theorem 1.2 of \cite{dembo1996tail}, states this precisely.
\begin{thm}\label{thm:poszerodrift_ldp}
Suppose that $\ev[\rho]<1$, but $\rho_{\text{max}}=1$ and $\pr(\omega_0=\frac{1}{2})>0$. Then for any open $G\subset(0,v_\alpha)$ which is separated from $v_\alpha$
\[
-\infty<\liminf_{n\rightarrow\infty}n^{-1/3}\log\pr(n^{-1}X_n\in G)\leq\limsup_{n\rightarrow\infty}n^{-1/3}\log\pr(n^{-1}X_n\in G)<0.
\]
\end{thm}

\subsection{Applications of RWRE for MERW}
\begin{definition} We define the right front and the left front of the MERW by \bae
M_X(n)&=\max\{X_m:1\leq m\leq n\}\\
M_Y(n)&=\max\{Y_m:1\leq m\leq n\}\\
R_n&=\max\{M_X(n),M_Y(n)\}\\
L_n&=\min\{M_X(n),M_Y(n)\} \eae
The particle associated to the right front at time $n$ is $X_n$ if $R_n=M_X(n)$ and it is $Y_n$ if $R_n=M_Y(n)$. If $R_{n}=M_X(n)=M_Y(n)$ we choose $X_n$ to be the right front arbitrarily.
\end{definition}
The main use of RWRE in our paper is captured in the next lemma. The next lemma shows that the number of doubly eaten cookies the particle associated to the left front sees, stochastically dominates an i.i.d random environment.
\begin{definition}
Let $\xi:\Z\rightarrow\{1/2,p\}$ be defined as follows: for every $x\in\Z$ there exists the first time $n$ such that $L_n=x+1$. If at time $n$ both walks ate the cookies in position $x$, $\xi(x)=p$, else $\xi(x)=1/2$.
We call $\xi$ the intersection environment. \end{definition}
\begin{definition}
Let $\om$ and $\om'$ be to environment defined on the same space. We say $\om$ dominates $\om'$ is there is a coupling $Q$ on the product space such that $Q(\om\ge \om')=1$, where $\om\ge\om'$ stands for point wise domination.
\end{definition}
\begin{lem}\label{lem:environmentdef}
The intersection environment $\xi$ dominates an i.i.d random environment with
\[
\omega(x)=\left\{\begin{array}{ll}
                        \frac{1}{2}~w.p~1-(1-p)^2\\
                        p~w.p~(1-p)^2
                \end{array}\right.
.\]
\end{lem}
\begin{proof}
Every time a walker turns left to some positive position, we know it ate all the cookies in that position. Our goal is to build a random environment that controls the movement of both walkers. We do it by building an environment that is an intersection of some of the doubly eaten cookie sites of both the walkers. We build the environment in the following manner: Consider some time $n$ such that $L_n>x+1$. The probability both of the walkers turn left on their first visit to $x+1$ is greater than $(1-p)^2$. Define a coupling, $U_X(x),U_Y(x)\sim \mathcal{U}(0,1)$ i.i.d. On the first visit to $x+1$, $X$ turns left if $U_X(x)\geq\omega_1(x+1)$ and $Y$ turns left if $U_Y(x)\geq\omega_2(x+1)$. $\omega(x)=p$ if $U_X(x)<(1-p)$ and $ U_Y(x)<(1-p)$ otherwise $\omega(x)=\frac{1}{2}$. By this coupling we know that if $\omega(x)=p$, both of the walks turned left at their first visit to $x+1$.
\end{proof}

For the random environment described in lemma \ref{lem:environmentdef}, let us make some calculations we will use in later sections.
\bae
\ev[\rho_x]&=\ev\left[\frac{1-\omega(x)}{\omega(x)}\right]=\frac{(1-p)^3}{p}+p(2-p)<1\\
\ev\left[\frac{1}{\omega_0}\right]&=2p(2-p)+\frac{(1-p)^2}{p}
.\eae
By theorem \ref{thm:rwrespeed} the speed of the RWRE induced by the MERW is
\bae
\lim_{n\rightarrow\infty}\frac{\xi_n}{n}=\frac{1-\ev[\rho_0]}{\ev\left[\frac{1}{\omega_0}\right]}=\frac{(1-p)^2(2p-1)}{2p^2(2-p)+(1-p)^2}>0.
\eae

\begin{rmk}
At this point it is not yet clear that the MERW even has a positive speed. This we prove in the next section.
\end{rmk}

\section{Positive speed}
\subsection{Transience of the walks}
In this section we use notations and ideas introduced by Martin P.W. Zerner in \cite{ZERMERW}.
\begin{definition}
Denote by $TRS\xi$ the event $\{\forall z\in\mathbb{Z},\exists n\forall m>n, \xi_m\neq z\}$. A process $\{\xi_n\}_{n=1}^\infty$ is transient if $\pr(TRS\xi)=1$.
\end{definition}
\begin{definition}
Let $z\in\mathbb{Z},z>0$, define the events
\bae
RECX(z)&=\{X_n=z\quad {\rm i.o}\}\\
RECY(z)&=\{Y_n=z\quad {\rm i.o}\}
.\eae
For a given $z$, if $RECX(z)$ occurs, we say that $\{X_n\}_{n=1}^\infty$ is $z$-recurrent. Same for $\{Y_n\}_{n=1}^\infty$.
\end{definition}
\begin{prop}\label{prop:trans221}
The MERW is transient for $p=1$.
\end{prop}
We prove this proposition by a sequence of lemmas.
\begin{lem}
For every $z'>z$, $\pr(RECX(z')| RECX(z))=1$.
\end{lem}
\begin{proof}
Every time $\{X_n\}$ reaches $z$ it has some positive probability to reach $z'>z$. Thus by standard arguments
$\{X_n\}$ is $z'$-recurrent a.s.
\end{proof}
\begin{lem}\label{lem:recyrecxzero}
Let $z\in\mathbb{Z},z>0$ then $\pr(RECY(z)\cap RECX(z))=0$.
\end{lem}
\begin{proof}
If $\{X_n\}$ is $z$-recurrent, it will eat all the
cookies between $z$ and $z+1$ a.s. Thus $\{Y_n\}$ can't return to
$z$ i.o a.s.
\end{proof}

\begin{cor}
For every $z\in\mathbb{Z}$, $\pr(TRSY|RECX(z))=1$
\end{cor}
\begin{proof}
\bae
&\pr(TRSY\cap RECX(z))\\&=\pr\bigg[TRSY\cap RECX(z)\cap\left(\cup_{z'}RECY(z')\biguplus(\cup_{z'}RECY(z'))^c\right)\bigg]\\
&=\pr[RECX(z)\cap\cup_{z'}RECY(z')]+P[RECX(z)\cap(\cup_{z'}RECY(z'))^c]\\
&=\pr[RECX(z)]
\eae
where $\biguplus$ stands for disjoint union and the second equality follows from the fact that by Lemma \ref{lem:recyrecxzero}, \bae&\pr[RECX(z)\cap\cup_{z'}RECY(z')\cap TRSY]\leq\\&\pr[RECX(z)\cap\cup_{z'}RECY(z')]=0.\eae
\end{proof}

\begin{cor}\label{cor:eithertrans}
$\pr(TRSY|TRSX^c)=1$.
\end{cor}
\begin{proof}
Given $TRSX^c$ there exists some $z$ such that $RECX(z)$ occurs. In that case we saw $TRSY$ occurs a.s.
\end{proof}

\begin{lem}\label{lem:bothtrans}
Let $z\in\mathbb{Z}$, $\pr(RECY(z)|TRSX)=0$.
\end{lem}
\begin{proof}
Let $\tau=\sup\{m>0|X_m=\xi\}$ for some $\xi>z$. For infinitely many
times, greater than $\tau$, $\{Y_n\}$ will reach $z$ and $\{X_n\}$
will be to the right of $\xi$. If $\{X_n\}$ is to the right of
$\max\{z'\in\mathbb{Z}:\exists k\leq n, Y_k=z'\}$, there are no
eaten cookies in its position and $\{X_n\}$ has probability
$\frac{1}{2}$ to go left. If $\{X_n\}$ is to the left of
$\max\{z\in\mathbb{Z}:\exists k\leq n, Y_k=z\}$, there are no
cookies left between $z$ and $\max\{z\in\mathbb{Z}:\exists k\leq
n, Y_k=z\}$. In this case $\{X_n\}$ will go right until it reaches
an uneaten cookie, and again it has a probability $\frac{1}{2}$ of
going left. Thus by the Borel-Cantelli lemma $\{X_n\}$ will eat
all the cookies in a site $\xi'>z$, after $\{Y_n\}$ passes $\xi'$
it will never return to $z$ which contradicts the $z$-recurrence
of $\{Y_n\}$.
\end{proof}
\begin{cor}
$\pr(TRSY^c|TRSX)=0$.
\end{cor}
\begin{proof}
$\pr(TRSY^c|TRSX)\leq\sum_z\pr(RECY(z)|TRSX)=0$.
\end{proof}
\begin{cor}
$\pr(TRSY|TRSX)=1$.
\end{cor}
\begin{proof}
$1=\pr(TRSY|TRSX)+\pr(TRSY^c|TRSX)=\pr(TRSY|TRSX)$.
\end{proof}

\begin{proof}[Proof of proposition \ref{prop:trans221}]
\bae
\pr(TRSY)&=\pr(TRSY|TRSX)\pr(TRSX)+\pr(TRSY|TRSX^c)\pr(TRSX^c)\\&=\pr(TRSX\cup TRSX^c)=1.
\eae
\end{proof}

\begin{prop}\label{prop:MERWtransb}
The MERW is transient for all $\frac{1}{2}<p<1$.
\end{prop}
\begin{proof}
Without loss of generality we prove $X_n$ is transient. By Lemma \ref{lem:environmentdef} the intersection environment dominates an i.i.d RWRE $\xi$. For some realization of $\xi$ let $\sigma^{i}_{-k}$ be i.i.d, with distribution $\tau_{-k}$, and $\sigma_1^i$ be i.i.d with distribution $\tau_1$, such that $\sigma^i_{-k}$ and $\sigma_1^i$ are coupled by the same random walk on $\xi$. Denote by $B[k,m]=\cap_{i=1}^{m}\{\sigma^i_1<\sigma_{-k}^i\}$.
By Corollary \ref{cor:backtracking},
\bae\label{eq:probofeventb}
\pr(B[k,m])\geq(1-2(\ev[\rho_0])^k)^m.
\eae
We obtain by first order approximations that
\bae\label{eq:backseveraltrialsbound}
\pr(B[n,n^4]^c)\leq1-(1-2(\ev[\rho_0])^n)^{n^4}\le 2n^4 \ev[\rho_0]^n,
\eae
thus
\bae
\pr(B[n,n^4]^c\text{ i.o})=0.
\eae

Since $L_n$ stochastically dominates the maximum of SRW, it is asymptotically larger than $n^\frac{1}{4}$.
Denote by $G_n$ the event \[G_n=\left\{L_n>n^\frac{1}{4}\right\},\]
and let $M_n$ be the maximum of a SRW and $\tau_a$ the first hitting time of a SRW in the vertex $a$.
 Thus
\bae\label{eq:nnupbnf}
\pr\left(G_n^c\right)\le\pr\left(M_n\le n^\frac{1}{4}\right)=\pr\left(\tau_{n^{1/4}}\ge n\right)\le\frac{1}{\sqrt{n}},
\eae
where the last inequality is Markov's inequality. Thus
\[
\pr\left(G_{n^4}^c\text{ i.o}\right)=0
.\]

\bae
\pr(TRSY^c)\le\pr\left(X_n\text{ backtracks from }L_n\text{ to }0\text{ i.o}\right).
\eae
Since between times $n^4$ to time $(n+1)^4$ there can be at most $n^4$ backtrack attempts,
\bae
\pr\left(X_n\text{ backtracks from }L_n\text{ i.o}\right)&=\pr\left(X_n\text{ backtracks from }L_n\text{ i.o}\right|\{G_{n^4}^c\text{ i.o}\}^c)\\
&\le\pr(B[n,n^4]^c\text{ i.o})=0.
\eae
\end{proof}


\subsection{Tightness bounds}\label{sec:tight}
\begin{definition}
We say that a sequence of random variables $\{Z_n\}_{n=1}^\infty$ is tight, if for
every $\varepsilon>0$ there exists some $M>0$ such that for all $n$
\[
\pr(|Z_n|>M)<\varepsilon.
\]
\end{definition}
In this section we begin the work to show that the process $\{X_n-Y_n\}_{n=1}^\infty$ is tight, and give bounds that will be useful in later sections.

We start with a simple lemma stating that both walks and both fronts are relatively close to each other.

\begin{lem}\label{lem:sqrt} For all $n$ large enough,
\begin{enumerate}
\item
$\prob(M_X(n)-X(n)>n^{0.6})<e^{-n^{0.1}}$
\item
$\prob(|Y(n)-X(n)|>n^{0.6})<e^{-n^{0.1}}$
\end{enumerate}
\end{lem}

\begin{proof}
We first show that if $n$ is sufficiently large then
\begin{equation}\label{eq:x_left_front}
\prob(L(n)-X(n)>n^{0.5})<e^{-n^{0.1}}.
\end{equation}
Equivalently,
\begin{equation}\label{eq:y_left_front}
\prob(L(n)-Y(n)>n^{0.5})<e^{-n^{0.1}}.
\end{equation}

\eqref{eq:x_left_front} follows from Lemma \ref{lem:rwrebacktrack} and Lemma \ref{lem:environmentdef}.
Indeed, $L(n)-X(n)>n^{0.5}$ only if there exists a backtrack excursion up to time $n$ of length larger than
$\sqrt{n}$. Since there are at most $n$ such excursions, and we know that for each of those the probability of being
larger than $\sqrt{n}$ is exponentially small in $\sqrt{n}$, we get that for some constant $C$,
\begin{equation*}
\prob(L(n)-X(n)>n^{0.5})<ne^{-C\sqrt{n}}<e^{-n^{0.1}}
\end{equation*}
for $n$ large enough.

To finish the proof of the lemma, we need to show that
\begin{equation}\label{eq:between_fronts}
\prob(R(n)-L(n)>n^{0.6})<e^{-n^{0.1}}.
\end{equation}

To this end, we show that $R(n)-L(n)$ is dominated by the front of a simple random walk, and then
\eqref{eq:between_fronts} will follow immediately by Azuma's inequality and the reflection principle.

To see the domination, all we need to note is that $R(n)-L(n)\leq \max\{\max(X_k,Y_k)-L(k)^+:k=1,\ldots,n\}$ and that the evolution of
$\max(X_n,Y_n)-L(n)$ is dominated by that of a simple random walk.

\end{proof}

We define the event $A_n$ that the gaps did not grow faster than desired: For every $n$ i.e.
\[
A_n=
\{
\forall_{1\leq k \leq n} M_X(k)-X(k)<n^{0.6} \mbox{ and } |Y(k)-X(k)|<n^{0.6}
\}.
\]

Then the probability of $A_n^c$ decays stretched exponentially with $n$.

\begin{definition}
We say that $t$ is a {\em fresh epoch} if $R(t)=X(t)=Y(t)=R(t-1)+1$.
We denote by $F(t)$ the event that $t$ is a fresh epoch.
\end{definition}
The next lemma governs the probability of appearance of fresh epochs.

\begin{lem}\label{lem:fresh}
There exists a constant $C$ such that
\[
\prob
\left(\left.
\bigcup_{t=n}^{n+n^{0.7}}F(t)
\right|
X(1),\ldots,X(n-1),Y(1),\ldots,Y(n-1),A_{2n}
\right)>1-\exp(-Cn^{0.1})\text{ a.s}
.\]
\end{lem}

\begin{proof}
The lemma will follow immediately once we prove that there exists a constant $C$ such that, under the same conditioning, the probability that there exists
a fresh epoch between $n$ and $n+Cn^{0.6}$ is bounded away from zero.

To this end we define the following sequence of stopping times:

We define $M_0 := R(n)$

Let $Z\in\{X,Y\}$ be the walk that corresponds to $R(n)$, and let $W$ be the other walk. Then $T_0 := \min\{ t: W(t) = M_0\}$.

Now define $M_1 := R(T_0)$.

We continue to recursively define $M_k$ and $T_k$ for all values of $k$.

We now need to make estimates on the random variables $M_k$ and $T_k$.

Define the event $U_k$:
\[
U_k := \{ T_{k+1}-T_k<n^{0.65^k}  \mbox{ and }
R(T_{k+1})-R(T_{k}) < (T_{k+1}-T_{k})^{0.6}     \}.
\]
We next estimate the probability of the event $U_k$.

\begin{claim}\label{claim:uk}
\[
\prob(U_k^c|U_1,\ldots,U_{k-1})<e^{-n^{0.1\cdot 0.65^{k}}}
\]
\end{claim}

We now assume Claim \ref{claim:uk} (We will prove it later) and finish the proof of Lemma \ref{lem:fresh}.

Let $M$ be large but fixed, and let $k_0=k_0(n)$ be the smallest number so that $n^{0.65^{k_0}}<M$.
Then by Claim \ref{claim:uk}, $\prob(U_1\cap U_2\cap\ldots\cap U_{k_0})$ is bounded away from zero.
Conditioned on the event $U_1\cap U_2\cap\ldots\cap U_{k_0}$, the probability a fresh epoch  existence
within the time frame $[n,n+Cn^{0.6}]$ is at least $0.25^M$ (probability for both the walks to move $M$ positions to the right is bigger than $0.25^M$).

\end{proof}
\begin{proof}[Proof of Claim \ref{claim:uk}]
Let $Z\in\{X,Y\}$ be so that $R(T_k)=M_Z(T_k)$, and let $W$ be the other element of $\{X,Y\}$.

Conditioned on the occurrence of $U_{k-1}$, the distance between $R(T_k)$ and $W(T_k)$ is
bounded by $Cn^{0.6^k}$, and therefore, by Lemma \ref{lem:environmentdef} and Theorem \ref{thm:poszerodrift_ldp} the probability that $T_{k+1}-T_k>n^{0.65^k}$ decays stretched-exponentially with
$n^{0.6^k}$. To control the probability that $R(T_{k+1})-R(T_k) \geq (T_{k+1} - T_k)^{0.6}$ we note that between times $T_k$ and $T_{k+1}$ the right front is dominated by the front of a reflected SRW.
\end{proof}

\subsection{Regeneration times}\label{subseq:regtime}

\begin{definition}\label{def:regtime}
We say that $t$ is a {\em regeneration time} if
\begin{enumerate}
\item $X_t=Y_t$.
\item For every $s>t$, $X_s>X_t$ and $Y_s>Y_t$.
\item For every $s<t$, $X_s<X_t$ and $Y_s<Y_t$.
\end{enumerate}
\end{definition}

Note that every regeneration time is a fresh epoch, but not vice versa.
On the other hand, we have the following two facts, which show that many fresh epochs are
indeed regeneration times.

\begin{lem}\label{lem:prob_reg}
Let $F(t)$ be the event that $t$ is a fresh epoch, and let $V(t)$ be the event that $t$ is a regeneration time. There exists a fixed $\rho=\rho(p)$, which is strictly between zero and one, such that for every $t$,
\[
\prob\big(V(t)|F(t)\big)=\rho.
\]
\end{lem}

\begin{proof}
Denote by $\rho=\pr(V(0))$. Notice that the event that a fresh epoch is a regeneration time is dependent only on the trajectory of the MERW to the right of the MERW at the fresh epoch. Thus for every $t$, $\prob\big(V(t)|F(t)\big)=\pr(V(0))$.  By Proposition \ref{prop:MERWtransb} the MERW is transient and thus visits the origin finitely many times a.s. Clearly, $\rho<1$ since there is a probability greater than $1-p^2$ for at least one of the walkers to turn left and thus admit $V(0)^c$. We are left with proving $\rho>0$. Let $M\in\BN$, then with positive probability, $\alpha>0$, after $3M$ steps, the intersection environment of the MERW in $[0,M]$ is constant $p$. By the same argument as in the proof of Proposition \ref{prop:MERWtransb}
\bae
\pr(V(0))&>\alpha\pr(\forall n>3M\text{ no backtracking from }L_n\text{ to }0)\\
&\ge\alpha\pr\left(\cap_{n>3M}B[n,n^4]\right)\ge1-\sum_{n=3m}^{\infty}\pr(B[n,n^4]^c)>0,
\eae
where the last inequality holds for $M$ large enough by \eqref{eq:backseveraltrialsbound}.

\end{proof}

\begin{lem}\label{lem:mix_prob_reg}
For every $n$, let $t_n$ be the first fresh epoch after time $n$. For every $k<n^{0.2}$ let $h(k)=n+kn^{0.8}$. Let
\bae B_n&=\bigcap_{k=1}^{n^{0.2}}\left\{\exists\text{at least } n^{0.1}\text{ fresh epochs between time }h(k)\text{ and }h(k+1)\right\}\\
Q(k)&=V(t_{h(k)})
\eae Then for every $n$ big enough
\[
\pr\left(Q(k)\left| \cap_{j<k}Q(j)^c,B_n,A_{2n}\right.\right)>\rho/2.
\]
\end{lem}
\begin{proof}
Let $H(k)=\left\{\exists t\in\left[t_{h(k-1)}+1,h(k)\right], \min\{X_t,Y_t\}=X_{t_{h(k-1)}}\right\}$, thus \bae
\pr\left(Q(k)\left| \cap_{j<k}Q(j)^c,B_n,A_{2n}\right.\right)&=\pr\left(Q(k)\cap H(k)\left| \cap_{j<k}Q(j)^c,B_n,A_{2n}\right.\right) \\
&+\pr\left(Q(k)\cap H(k)^c\left| \cap_{j<k}Q(j)^c,B_n,A_{2n}\right.\right)\\
&\geq \pr\left(Q(k)\left| \cap_{j<k}Q(j)^c,B_n,A_{2n},H(k)\right.\right)\pr(H(k)| \cap_{j<k}Q(j)^c,B_n,A_{2n})
.\eae
Given $B_n$ the probability of $H(k)^c$ is stretch exponentially small, i.e.
\bae\label{eq:hkbound}
\pr(H(k)^c|\cap_{j<k}Q(j)^c,B_n)<e^{-cn^{0.1}}
.\eae
Under $H(k)$ the probability of a regeneration at $t_{h(k)}$ is independent of $\cap_{j<k}Q(j)^c$.
By lemma \ref{lem:prob_reg} and \eqref{eq:hkbound}, for $n$ large enough
\[
\pr\left(Q(k)\left| \cap_{j<k}Q(j)^c,B_n,A_{2n},H(k)\right.\right)\pr(H(k)| \cap_{j<k}Q(j)^c,B_n,A_{2n})>\rho/2.
\]
.\end{proof}

\begin{thm}\label{thm:regtime}
There exists a constant $\xi>0$ such that for all large enough $n$,
\[
\prob\left(\bigcap_{k=1}^nV^c(k)\right)<e^{-n^\xi}.
\]
\end{thm}
\begin{proof}
\bae
\pr\left(\bigcap_{k=1}^nV^c(k)\bigg|B_n,A_{2n}\right)&\leq\pr\left(\bigcap_{k=1}^{n^{0.2}}Q(k)^c\bigg|B_n,A_{2n}\right)\\
&=\prod_{j=1}^{n^{0.2}}\pr\left(Q(j)^c\bigg|B_n,A_{2n},\bigcap_{k=1}^{j}Q(k)^c\right)    \\
&\leq (1-\rho/2)^{n^{0.2}}.
\eae
Thus
\bae\label{eq:regstretchexp}
\pr\left(\bigcap_{k=1}^nV^c(k)\right)\leq \pr\left(\bigcap_{k=1}^nV^c(k)\bigg|B_n,A_{2n}\right)+\pr((B_n\cap A_{2n})^c)
\eae
and the RHS of \eqref{eq:regstretchexp} decays stretch exponentially.

\end{proof}

\begin{cor}
Let $\tau$ be a regeneration time then $\ev[X_\tau]<\infty$.
\end{cor}
\begin{proof}
$|X_\tau|<\tau$ thus by theorem \ref{thm:regtime}
\[
\ev[X_\tau]<\ev[\tau]<\infty.
\]
\end{proof}

\begin{lem}\label{lem:regiid}
Let $\tau_1<\tau_2<\ldots$ be all the regeneration times, then $\{\tau_{i+1}-\tau_i\}_{i\in\mathbb{N}}$ are i.i.d.
\end{lem}
\begin{proof}
Let $\tau<\sigma$ be a two regeneration times. $\{X_{\tau+t}\},\{Y_{\tau+T}\}$ are independent of all the trajectories to the left of $\{X_\tau\}$, thus $\sigma-\tau$ is independent of $\tau$. Since after every regeneration time the process runs on the same environment the regeneration differences $\{\tau_{i+1}-\tau_i\}_{i\in\BN}$ are identically distributed.
\end{proof}

Using Theorem \ref{thm:regtime}, Lemma \ref{lem:regiid} and the law of large
numbers we have
\begin{equation}\label{eq:speedregrepresentation}
v\stackrel{{\rm
def}}{=}\lim_{i\rightarrow\infty}\frac{X_{\tau_i}}{\tau_i}=\lim_{i\rightarrow\infty}\frac{X_{\tau_i}}{i}\cdot\frac{i}{\tau_i}\stackrel{LLN}{=}\frac{\ev[X_{\tau_{2}}-X_{\tau_1}]}{\ev[\tau_{2}-\tau_1]}
\end{equation}
For every $n$ there exists some $i(n)$ such that
$\tau_{i(n)}<n<\tau_{i(n)+1}$ and satisfy $X_{\tau_{i(n)}}\leq X_n
\leq X_{\tau_{i(n)+1}}$. thus
\begin{equation}v{\underset{\infty\leftarrow n}{\longleftarrow}} \frac{X_{\tau_{i(n)}}}{\tau_{i(n)}}\cdot \frac{\tau_{i(n)}}{n}\leq\frac{X_n}{n}\leq
\frac{X_{\tau_{i(n)+1}}}{\tau_{i(n)+1}}\cdot\frac{\tau_{i(n)+1}}{n}{\underset{n\rightarrow\infty}{\longrightarrow}}
v
\end{equation}
This yields the main theorem of this section
\begin{thm}\label{thm:converspeed}
For every $\frac{1}{2}<p<1$, $\{X_n\}$ and $\{Y_n\}$ have positive speed a.s. i.e
\[
\lim_{n\rightarrow\infty}\frac{X_n}{n}=\lim_{n\rightarrow\infty}\frac{Y_n}{n}=v>0,~{\rm
a.s}.
\]
\end{thm}
\begin{prop}
For $\frac{1}{2}\leq p<1$ The speed $v$ is smaller than $\frac{1}{1+\frac{2}{2p-1}}$.
\end{prop}
\begin{proof}
Let $\nu_i$ be the number of drifts $\{X_n\}$ left in the interval $X_{\tau_{i+1}}-X_{\tau_i}$ after time $\tau_{i+1}$ and $\om\in\{0,1\}^\BZ$ be 1 in the positions where $\{X_n\}$ left a drift and $0$ elsewhere. By the discussion above $\nu_i$ are i.i.d. Let $\lambda=\lim_{i\rightarrow\infty}\frac{\nu_i}{i}$. Note that $Y_n$ has a smaller drift at any position than the drift left by $X_n$ after the next regeneration time. Thus the speed of $Y_n$ which equals the speed of $X_n$ is smaller than the speed of a RWRE on the environment $\om$. By \cite{procaccia2012need} Theorem 1.3, $v\le(2p-1)\lambda$. But each drift left requires at least two moves, thus $v\le1-2\lambda$. Combining the two inequalities we get the proposition.

\end{proof}
\section{Markov approximation}
In this section we present a family of Markov processes whose speed approximates the speed of MERW, and has a simple yet hard to calculate representation.
\subsection{The model for the $k$'th process}
Consider two MERW with transition probabilities as in \eqref{eq:modelrule} truncated at $R_n-k$, i.e if $X_n$ or $Y_n$ reach a distance $k$ (even) from the right front, they turn right with probability one.
The state space of this Markov process $Z_n^k$ is contained in $\{0,1,2\}^{2k}\times\{0,...,k\}^2$. In each of the $k$ positions, one has to know the number of times $X_n$ and $Y_n$ visited and one has to know the current position of the two walks. Given that knowledge, the next state is calculated by \eqref{eq:modelrule} and the constraint. The environment changes according to the movement of the walks. As an example consider the case $k=2$, denote by a square one process and by a circle the other.

\begin{figure}[H]
\begin{center}
\includegraphics[width=0.6\textwidth]{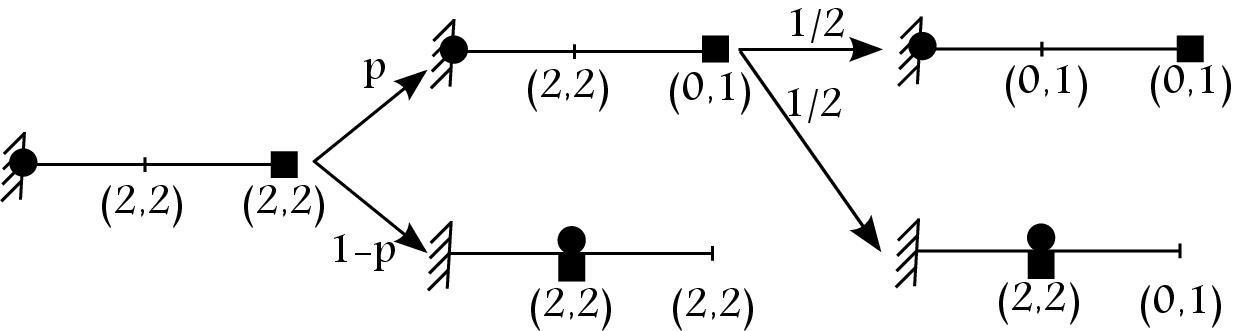}
\end{center}
\end{figure}

Next we prove convergence of the Markov chain speed to the MERW speed. By \eqref{eq:speedregrepresentation}, it is enough to control the change of regeneration time and distance expectation.
\begin{lem}\label{lem:badreg}
Let $A_k=\{\exists \tau_1\leq n\leq \tau_2:X_n<R_n-k\}$, then $\pr(A_k)\underset{k\rightarrow\infty}{\longrightarrow}0$.
\end{lem}
\begin{proof}
Let $B_k=\{\tau_2-\tau_1\geq k\}$, from Theorem \ref{thm:regtime}, $\pr(B_k)\underset{k\rightarrow\infty}{\longrightarrow}0$, moreover for $p\in[p_1,p_2]$ for some $1/2<p_1<p_2<1$ the convergence is uniform. If $X_n<R_n-k$ it will take more than $k$ steps for $X_n$ to reach a fresh point. Thus $A_k\subset B_k$ and $\pr(A_k)\underset{k\rightarrow\infty}{\longrightarrow}0$.
\end{proof}

\begin{lem}\label{lem:unconv}
Let $v_k$ be the speed of the $k$'th Markov process. For all $1/2<p_l\leq p\leq p_h<1$, $\lim_{k\rightarrow\infty}v_k(p)=v(p)$ uniformly.
\end{lem}
\begin{proof}
We calculate the speed change caused by truncation, by conditioning on the event $A_k$.
\bae\label{eq:errorexp}
\ev[(\tau_2-\tau_1)\ind_{A_k}]&\leq\sqrt{\ev[(\tau_2-\tau_1)^2]}\sqrt{\ev[\ind_{A_k}^2]}\\
&=\sqrt{\pr(A_k)}\sqrt{\ev[(\tau_2-\tau_1)^2]}\underset{k\rightarrow\infty}{\longrightarrow}0.
\eae
Where the inequality is by Cauchy-Schwarz and the limit is due to Lemma \ref{lem:badreg}. Let $Z_n$ be one of the truncated MERW, $\sigma_n$ its regeneration times and $X_n$ be one of the MERW,
\bae\label{eq:vkcalc}
v_k&=\frac{\ev[Z_{\sigma_2}-Z_{\sigma_1}]}{\ev[\sigma_2-\sigma_1]}\\
&=\frac{\ev[(Z_{\sigma_2}-Z_{\sigma_1})\ind_{A_k}]+\ev[(Z_{\sigma_2}-Z_{\sigma_1})\ind_{A_k^c}]}{\ev[(\sigma_2-\sigma_1)\ind_{A_k}]+\ev[(\sigma_2-\sigma_1)\ind_{A_k^c}]}\\
&=\frac{\ev[(Z_{\sigma_2}-Z_{\sigma_1})\ind_{A_k}]+\ev[(X_{\tau_2}-X_{\tau_1})\ind_{A_k^c}]}{\ev[(\sigma_2-\sigma_1)\ind_{A_k}]+\ev[(\tau_2-\tau_1)\ind_{A_k^c}]}\\
&=\frac{\ev[X_{\tau_2}-X_{\tau_1}]-\ev[(X_{\tau_2}-X_{\tau_1})\ind_{A_k}]+\ev[(Z_{\sigma_2}-Z_{\sigma_1})\ind_{A_k}]}{\ev[\tau_2-\tau_1]-\ev[(\tau_2-\tau_1)\ind_{A_k}]+\ev[(\sigma_2-\sigma_1)\ind_{A_k}]}
.\eae
Let
\bae\label{eq:varepsilonlimit}
\varepsilon_k=\max\{2\ev[(\tau_2-\tau_1)\ind_{A_k}],2\ev[(\sigma_2-\sigma_1)\ind_{A_k}]\}
\eae
By \eqref{eq:errorexp} $\lim_{k\rightarrow\infty}\varepsilon_k=0$. Combining \eqref{eq:vkcalc} and \eqref{eq:varepsilonlimit} yields,
\bae
v_k&\leq\frac{\ev[X_{\tau_2}-X_{\tau_1}]+\varepsilon_k}{\ev[\tau_2-\tau_1]-\varepsilon_k}=\frac{\ev[X_{\tau_2}-X_{\tau_1}]}{\ev[\tau_2-\tau_1]}\frac{1}{1-\frac{\varepsilon_k}{\ev[\tau_2-\tau_1]}}
+\frac{\varepsilon_k}{\ev[\tau_2-\tau_1]-\varepsilon_k}\\
&\leq \frac{\ev[X_{\tau_2}-X_{\tau_1}]}{\ev[\tau_2-\tau_1]}\left(1+\frac{2\varepsilon_k}{\ev[\tau_2-\tau_1]}\right)+\frac{\varepsilon_k}{\ev[\tau_2-\tau_1]-\varepsilon_k}\\
&=v+2\varepsilon_k\ev[X_{\tau_2}-X_{\tau_1}]+\frac{\varepsilon_k}{\ev[\tau_2-\tau_1]-\varepsilon_k}\underset{k\rightarrow\infty}{\longrightarrow}v.
\eae
For the lower bound
\bae
v_k&\geq\frac{\ev[X_{\tau_2}-X_{\tau_1}]-\varepsilon_k}{\ev[\tau_2-\tau_1]+\varepsilon_k}=\frac{\ev[X_{\tau_2}-X_{\tau_1}]}{\ev[\tau_2-\tau_1]+\varepsilon_k}
-\frac{\varepsilon_k}{\ev[\tau_2-\tau_1]+\varepsilon_k}\\
&=\frac{\ev[X_{\tau_2}-X_{\tau_1}]}{\ev[\tau_2-\tau_1]}\frac{1}{1+\frac{\varepsilon_k}{\ev[\tau_2-\tau_1]}}
-\frac{\varepsilon_k}{\ev[\tau_2-\tau_1]+\varepsilon_k}\underset{k\rightarrow\infty}{\longrightarrow}v.
\eae
\end{proof}

\subsection{What do we get from the Markov approximation?}
In this section we show some results one can obtain using the Markov approximation tool.
\begin{lem}\label{lem:vkcon}
The speed $v_k(p)$ is continuous on the interval $\left[\frac{1}{2},1\right]$.
\end{lem}
\begin{proof}
Let $D_n^k$ be the number of steps $X_n^k$ turned to the right minus the number of left turns. We can write the speed as $v_k=\sum_y\pi(y)\sigma(y)$, where $\pi(\cdot)$ is the stationary distribution and $\sigma(y)$ is the drift $X^{k}_n$ has in state $y$. To see this
\bae
v_k&=\lim_{n\rightarrow\infty}\frac{D_n^k}{n}=\lim_{n\rightarrow\infty}\frac{\ev\left[D_n^k\right]}{n}=\lim_{n\rightarrow\infty}\frac{1}{n}\sum_{i=0}^{n-1}\ev\left[D_{n+1}^k-D_n^k\right]\\
&=\lim_{n\rightarrow\infty}\frac{1}{n}\sum_{i=0}^{n-1}\sigma(X_n^k)=\sum_y\pi(y)\sigma(y).
\eae
Both $\sigma(y)$ and $\pi(y)$ are continuous function of $p$, since the sum is over a finite space state, $v_k$ is a continues function of $p$.
\end{proof}
\begin{thm}
The speed $v(p)$ is continuous on the interval $\left[\frac{1}{2},1\right)$.
\end{thm}
\begin{proof}
The continuity in the interval $(\frac{1}{2},1)$ is an easy conclusion from the uniform convergence of $v_k$ to $v$ derived in Lemma \ref{lem:unconv} and the continuity of $v_k$ is proved in lemma \ref{lem:vkcon}. For every $p$ we have an upper bound for the speed, $v(p)\leq 2p-1$. Taking the limit $p\rightarrow \frac{1}{2}$ we obtain continuity at $\frac{1}{2}$,  \[0\leq v\left(\frac{1}{2}\right)\leq \lim_{p\rightarrow\frac{1}{2}}2p-1=0.\]
\end{proof}
\subsection{Deterministic approximation of the speed}
Let $P_k(x,y)$ be the transition matrix of $Z^k_n$. It is easy to see that
\[\lim_{m\rightarrow\infty}\max_{x}\left| \sum_y\pr^m_k(x,y)\sigma(y)-v_k\right|=0.\]
This provides us with a deterministic algorithm  to calculate $v_k$. Given this calculation, by Lemma \ref{lem:unconv} we can find $p_1<p_2$ and $k$ large enough such that $v_k(p_1)>v_k(p_2)$ and $|v(p)-v_k(p)|<\frac{v_k(p_1)-v_k(p_2)}{10}$ for all $p\in[p_1,p_2]$ and thus attain the non-monotonicity of $v(p)$. We were not able to make the proposed calculations as its complexity is too high for today's computers.

\section{Open questions}
\begin{enumerate}
\item The base of our technique was to use RWRE, this only works for $p<1$. One important question that eludes us is: Is $v(p)$ continuous at $p=1$.
\item Consider a generalization of the MERW model, MERW$(m)$. $m$ is the initial number of symmetric cookies in each site. That is
\[
{\omega_1(z)\choose\omega_2(z)}={(\overset{m~{\rm times}}{\overbrace{1/2,\ldots,1/2}},p,p,p,...)\choose
(\underset{m~{\rm times}}{\underbrace{1/2,\ldots,1/2}},p,p,p,...)}
.\]
An interesting conjecture arise from simulations (See Figure \ref{fig:merw357}). For $m$ large enough $v_m(p)$ is monotone, where $v_m(p)$ is the speed of MERW$(m)$.

\begin{figure}[H]
\begin{center}
\includegraphics[width=0.6\textwidth]{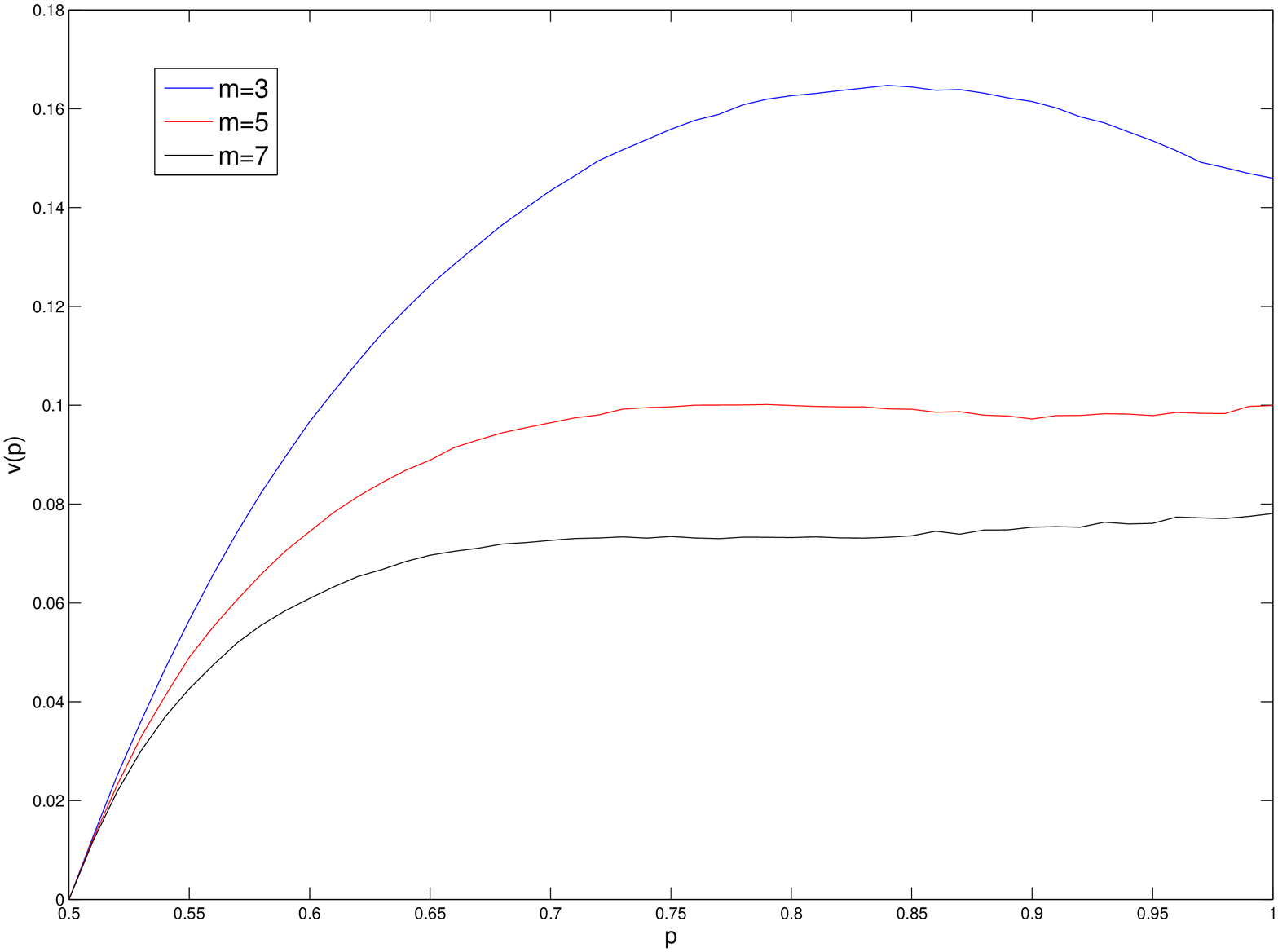}
\caption{Speed of MERW$(m)$, $m=3,5,7$\label{fig:merw357}}
\end{center}
\end{figure}
\end{enumerate}

\bibliography{ttp}
\bibliographystyle{alpha}

\end{document}